\nonstopmode
\documentclass[10pt, reqno]{amsart}
\usepackage{graphicx}
\usepackage{latexsym}
\usepackage{fancyhdr}
\usepackage{amsmath, amssymb, amsthm}
\usepackage[all]{xy}
\usepackage{pdflscape}
\usepackage{longtable}
\usepackage{rotating}
\usepackage{verbatim}
\usepackage{hyperref}
\hypersetup{
    linkcolor = blue,
    citecolor = blue,
    urlcolor  = blue,
    colorlinks = true,
}

\usepackage{etoc}
\etocsettocstyle
  {\medskip}          % no heading
  {\bigskip}

\etocsetstyle{section}                                          % start (before first entry)
  {}
  {\leavevmode\leftskip 0pt\relax}            % prefix (before each entry)
  {\etocnumber \etocname          % contents
   \dotfill\etocpage\par\smallskip}
  {}                                          % finish (after last entry)

\etocsetstyle{subsection}
  {}
  {\leavevmode\leftskip 1.5em\relax}
  {\etocnumber\ \etocname\dotfill\etocpage\par\smallskip}
  {}

\etocsettocstyle
  {\section*{\contentsname}\small}           % TOC heading
  {\bigskip}                                 % after TOC

\usepackage{cleveref}
\usepackage{subfigure}
\usepackage{mathrsfs}
\usepackage{mdwlist}
\usepackage{dsfont}

\usepackage{mathtools}
\mathtoolsset{showonlyrefs=true}

\usepackage{float}
\usepackage{color}
\usepackage{stmaryrd}
\usepackage{parskip}
\usepackage{orcidlink}

\usepackage{tkz-base}
\usepackage{pgfplots}
\pgfplotsset{compat=newest}

\usepackage{tkz-euclide}
\usepackage{etoolbox}

\definecolor{teal}{rgb}{0.0, 0.5, 0.5}

\newcounter{mnotecount}[section]

\newcommand{\rmnote}[1]{}%{\mnote{#1}}

\allowdisplaybreaks

\DeclareFontFamily{U}{mathb}{\hyphenchar\font45}
\DeclareFontShape{U}{mathb}{m}{n}{
      <5> <6> <7> <8> <9> <10> gen * mathb
      <10.95> mathb10 <12> <14.4> <17.28> <20.74> <24.88> mathb12
      }{}
\DeclareSymbolFont{mathb}{U}{mathb}{m}{n}
\DeclareFontSubstitution{U}{mathb}{m}{n}

\theoremstyle{plain}
\newtheorem*{theorem*}{Theorem}
\newtheorem{theorem}{Theorem}[section]
\newtheorem*{lemma*}{Lemma}
\newtheorem{lemma}[theorem]{Lemma}
\newtheorem*{assumption*}{Assumption}

\newtheorem*{proposition*}{Proposition}
\newtheorem{proposition}[theorem]{Proposition}
\newtheorem*{corollary*}{Corollary}

\newtheorem*{claim*}{Claim}

\newtheorem*{conjecture*}{Conjecture}

\newtheorem*{question*}{Question}

\newtheorem*{result*}{Result}

\theoremstyle{definition}
\newtheorem*{definition*}{Definition}

\newtheorem*{example*}{Example}

\newtheorem*{algorithm*}{Algorithm}
\newtheorem*{remark*}{Remark}
\newtheorem*{remarks*}{Remarks}

\newtheorem*{convention*}{Convention}

\Crefname{l}{Lemma}{Lemmas}    %declares the type of the environment for cleverref, use \label[<type>]{<label>}
\Crefname{p}{Proposition}{Propositions}
\Crefname{t}{Theorem}{Theorems}
\Crefname{c}{Corollary}{Corollaries}
\Crefname{r}{Remark}{Remarks}
\Crefname{d}{Definition}{Definitions}
\Crefname{e}{Example}{Examples}
\Crefname{q}{Question}{Questions}

\numberwithin{equation}{section}

\def\al{\alpha}

\def\de{\delta}
\def\ep{\epsilon}

\def\la{\lambda}

\def\rh{\rho}

\def\si{\sigma}

\def\ta{\tau}

\def\De{\Delta}

\def\La{\Lambda}

\def\C{\mathbb{C}}

\def\N{\mathbb{N}}

\def\R{\mathbb{R}}

\def\cA{\mathcal{A}}

\def\cL{\mathcal{L}}

\def\sol{\mathcal{S}}

\def\loc{\on{loc}}

\def\<{\langle}
\def\>{\rangle}
\renewcommand{\o}{\circ}

\def\ol{\overline}

\def\Lip{\on{Lip}}

\def\Hyp{\on{Hyp}}

\def\dd{\mathbf{d}}

\def\a#1{\left\llbracket{#1}\right\rrbracket}

\let\on=\operatorname

\newcommand{\sr}[1]%
{\ifmmode{}^\dagger\else${}^\dagger$\fi\ifvmode
\vbox to 0pt{\vss
 \hbox to 0pt{\hskip\hsize\hskip1em
 \vbox{\hsize3cm\raggedright\pretolerance10000
 \noindent #1\hfill}\hss}\vss}\else
 \vadjust{\vbox to0pt{\vss%
 \hbox to 0pt{\hskip\hsize\hskip1em%
 \vbox{\hsize3cm\raggedright\pretolerance10000%
 \noindent \small #1\hfill}\hss}\vss}}\fi%
}

\providecommand{\mapsfrom}{\kern.2em%
\setbox0=\hbox{$\leftarrow$\kern-.10em\rule[0.26mm]{0.1mm}{1.3mm}}\box0%
\kern.3em}

\title[Reparameterization and continuity of the solution map]
{Natural reparameterization and continuity of the solution map for polynomials}

\author[Adam Parusi\'nski and  Armin Rainer]
{Adam Parusi\'nski \orcidlink{0000-0002-8203-2561} and Armin Rainer \orcidlink{0000-0003-3825-3313}}

\address {Adam Parusi\'nski: Universit\'e C\^ote d'Azur,  CNRS,  LJAD, UMR 7351, 06108 Nice, France}
\email{adam.parusinski@univ-cotedazur.fr}

\address{Armin Rainer: Faculty of Mathematics and Geoinformation,
    Institute for Statistics and Mathematical Methods in Economics, E105-04,
TU Wien, Wiedner Hauptstraße 8, 1040 Vienna, Austria}
\email{armin.rainer@tuwien.ac.at}

\begin{document}

\begin{abstract}
    The optimal Sobolev regularity of the roots of a smooth curve of monic complex polynomials is $W^{1,q}$, 
    for $q \in [1,\frac{d}{d-1})$, where $d$ is the degree. This result is stable in the sense that the solution map 
    from $C^d$ coefficients to $W^{1,q}$ roots is continuous. We prove that, after a natural 
    Lipschitz reparameterization, the roots are Lipschitz and 
    the map from $C^d$ coefficients to reparameterization and reparameterized   
    roots is continuous with respect to the $W^{1,q}$ topology on the target spaces, for all $q \in [1,\infty)$. 
    The result is based on a convergence and reparameterization theorem for Almgren's $Q$-valued Sobolev functions.
\end{abstract}

\thanks{This research was funded in whole or 
    in part by the Austrian Science Fund (FWF) DOI 10.55776/PAT1381823.
For open access purposes, the authors have applied a CC BY public copyright license to any author-accepted manuscript version arising from this submission.}
   %Particular nonlinear operators (superposition, Hammerstein, Nemytskiĭ, Uryson, etc.)
%    14P10,      %Semialgebraic sets and related spaces
%	26B05, 	    %Continuity and differentiation questions
%	26B35,  	%Special properties of functions of several variables, HÃ¶lder conditions, etc.
%	26E10,  	%$C^\infty$-functions, quasi-analytic functions
%    26E25,      %Set-valued functions 
%    32B20,  	%Semi-analytic sets and subanalytic sets
%46E15}      %Banach spaces of continuous, differentiable or analytic functions
%  58C20,    %Differentiation theory (Gateaux, FrÃ©chet, etc.) on manifolds
%	58C25}  	%Differentiable maps
	%58B10,  	%Differentiability questions
	%58B25,  	%Group structures and generalizations on infinite-dimensional manifolds
	%58C07,  	%Continuity properties of mappings
	%58D05} 	    %Groups of diffeomorphisms and homeomorphisms as manifolds
%\dedicatory{dedicatory}
\date{September 30, 2026}

\maketitle

\setcounter{tocdepth}{1}
%\tableofcontents

%-----------------------------------------------------------------------------------------------------------------------
\section{Introduction}
%-----------------------------------------------------------------------------------------------------------------------

The optimal Sobolev regularity of the roots of monic complex polynomials with smooth coefficients was 
established in \cite{ParusinskiRainerAC,ParusinskiRainer15,Parusinski:2020aa}. If the coefficients are functions of class 
$C^{d-1,1}(\ol I)$, where $d$ is the degree of the polynomial and $I$ a bounded open interval, then each continuous root 
is absolutely continuous, bounded, and $\la' \in L^q(I)$, for all $q \in [1,\frac{d}{d-1})$. 
In \cite{Parusinski:2024ab}, 
stability of this result was proved, in the sense that the solution map which assigns to polynomials with coefficients 
in $C^d(\ol I)$ their unordered root tuple of Sobolev class $W^{1,q}$, for $q \in [1,\frac{d}{d-1})$, 
is continuous. 
In this note, we show that, by means of a natural Lipschitz reparameterization, these results extend to all values 
$q \in [1,\infty)$.

Since there is no canonical ordering of the complex roots, 
we work with the space of unordered $d$-tuples and
use Almgren's theory of multivalued Sobolev functions 
(see \cite{Almgren00}, \cite{De-LellisSpadaro11}). 
We view the symmetric product $\on{Sym}^d(\C) = \C^d/\on{S}_d$ as the 
complete metric space $(\cA_d(\C), \dd_2)$ of unordered $d$-tuples $[z]=[z_1,\ldots z_d]$ of complex numbers, 
where
\begin{equation}
    \mathbf d_2([z],[w]) := \min_{\si \in \on{S}_d} \|z - \si w\|_2 
    = \min_{\si \in \on{S}_d} \Big(\sum_{j=1}^d |z_j - w_{\si(j)}|^2 \Big)^{1/2}.
\end{equation}
There is a bi-Lipschitz embedding $\De : \cA_d(\C) \to \R^N$, for some $N = N(d)$, which can be used to define 
the Sobolev space
\begin{equation}
    W^{1,q}(I,\cA_d(\C)) := \{f : I \to \cA_d(\C) : \De \o f \in W^{1,q}(I,\R^N)\}, \quad q \in [1,\infty). 
\end{equation}
We call $\De$ an Almgren embedding. 
Pullback of the Sobolev norm on $W^{1,q}(I,\R^N)$ endows $W^{1,q}(I,\cA_d(\C))$ with a complete metric; 
the induced topology does not depend on the choice of $\De$. See \Cref{ssec:Q} for more details.

Each element $[\la_1,\ldots,\la_d] \in \cA_d(\C)$ corresponds in a one-to-one fashion to a 
complex polynomial 
\begin{equation} \label{eq:Pa}
    P_a(Z) = Z^d + \sum_{i=1}^d a_i Z^{d-i} =  \prod_{j=1}^d (Z-\la_j). 
\end{equation}    
Since the coefficient $a_i$ is (up to sign) the $i$-th elementary symmetric function of the roots, 
this defines a bijective map $a=(a_1,\ldots,a_d) : \cA_d(\C) \to \C^d$.
We denote by $\La : \C^d \to \cA_d(\C)$ the inverse map which sends (the coefficient vector of) a polynomial to its
unordered root tuple. The map $a$ is locally Lipschitz and $\La$ is 
locally $\frac{1}{d}$-H\"older (see \cite[Lemma 6.4]{Parusinski:2024ab}).

Let us consider curves of polynomials $P_a$, that is curves $a : I \to \C^d$, where
$I\subseteq \R$ is a bounded open interval. The unordered roots of $P_a$ are given by the composite 
$\La \o a : I \to \cA_d(\C)$.
We will call the nonlinear superposition operator $\sol : a \mapsto \La \o a$ the \emph{solution map} 
and write $\sol(a) = \La \o a = \La_a$ interchangeably.

\begin{theorem}[{\cite{ParusinskiRainer15,Parusinski:2024ab}}]\label[t]{t:know}
    Let $d \in \N_{\ge 2}$ and $q \in [1,\frac{d}{d-1})$. Then:
    \begin{enumerate}
        \item The map $\sol : C^{d-1,1}(\ol I,\C^d) \to W^{1,q}(I,\cA_d(\C))$ is well-defined and bounded.
        \item The map $\sol : C^{d}(\ol I,\C^d) \to W^{1,q}(I,\cA_d(\C))$ is continuous.
    \end{enumerate}
\end{theorem}

The map $\sol$ is not well-defined for $q \ge \frac{d}{d-1}$, even on polynomial curves in $\C^d$. In fact, the 
obstruction is already visible for
$Z^2 = x$,  $x \in (0,1)$, where the first derivative of the roots fails to be in $L^2((0,1))$. 
(In \Cref{t:know}, we assume $d \ge 2$ so that $\frac{d}{d-1}$ is well-defined, but the assertions are trivially true for $d=1$.)

Our main result is that the restriction on $q$ can be removed after a natural Lipschitz reparameterization 
associated with the rate $\ep + \|(\De \o \La_a)'\|_2$, where $\De$ is an Almgren embedding and $\ep>0$ 
is arbitrary. The parameter $\ep>0$ guarantees that the rate is strictly positive and hence the reparameterization is invertible on $I$. 
The potentially unbounded root velocity near collisions is absorbed into the reparameterization 
of time in a stable manner.

\begin{theorem} \label[t]{t:main1}
    Let $d \in \N_{\ge 1}$, $\ep>0$, and $\De : \cA_d(\C) \to \R^N$ an Almgren embedding.
    There is a reparameterization map $t=t(\ep,\De) : C^{d-1,1}(\ol I,\C^d) \to \Lip(\R,\ol I)$, 
    $a \mapsto t_a$, such that the following properties hold for all $a \in C^{d-1,1}(\ol I,\C^d)$:
    \begin{enumerate}
        \item $t_a : \R \to \ol I$ is surjective, increasing, and strictly increasing on $t_a^{-1}(I)$;
        \item $t_a$ is globally Lipschitz with Lipschitz constant $\Lip(t_a) \le \frac{1}{\ep}$; 
        \item the composite $\La_a \o t_a : \R \to \cA_d(\C)$ is globally Lipschitz with $\Lip(\La_a \o t_a) \le 1$.
    \end{enumerate}
    The induced map 
    \begin{align} \label{eq:rs}
        C^d(\ol I,\C^d) &\to W^{1,q}_{\on{loc}}(\R) \times  W^{1,q}_{\on{loc}}(\R,\cA_d(\C)) 
         \\
        a &\mapsto (t_a, \La_a \o t_a)
    \end{align}
    is continuous, for all $q \in [1,\infty)$.
\end{theorem}

While the construction of $t$ depends on the specific choice of $\ep$ and $\De$,
the continuity of the map in \eqref{eq:rs} is independent. 

This continuity is a consequence of the following convergence result, \Cref{t:seq}: $C^d$ convergence of the 
coefficients implies uniform convergence as well as $W^{1,q}$ convergence, for all $q \in [1,\infty)$, 
of the reparameterization and the reparameterized roots.

\begin{theorem} \label[t]{t:seq}
    Let $d \in \N_{\ge 1}$, $\ep>0$, $\De : \cA_d(\C) \to \R^N$ an Almgren embedding, and $t=t(\ep,\De)$ 
    the reparameterization map from \Cref{t:main1}.
    If $a_n \to a$ in $C^d(\ol I,\C^d)$, then:
    \begin{enumerate}
        \item $t_{a_n} \to t_a$ uniformly on $\R$ and in $W^{1,q}_{\loc}(\R)$ for all $q \in [1,\infty)$.
        \item $\La_{a_n} \o t_{a_n} \to \La_a \o t_a$ uniformly on $\R$ and in $W^{1,q}_{\loc}(\R,\cA_d(\C))$ for all $q \in [1,\infty)$.
        \item Assume that $\la_n : I \to \C^d$ is a continuous parameterization of the roots of $P_{a_n}$ and $\la_n$ converges in $C^0(\ol I,\C^d)$ 
            to a continuous parameterization $\la$ of the roots of $P_a$. 
            Then, for each $\ep >0$, there exist $t,t_n \in \Lip(\R,\ol I)$ with Lipschitz constants $\le \frac{1}{\ep}$ such that 
            $\la\o t, \la_n \o t_n \in \Lip(\R,\C^d)$ with Lipschitz constants $\le 1$ and
            \begin{itemize}
                \item $t_n \to t$ uniformly on $\R$ and in $W^{1,q}_{\loc}(\R)$ for all $q \in [1,\infty)$,
                \item $\la_{n} \o t_n \to \la \o t$ uniformly on $\R$ and in $W^{1,q}_{\loc}(\R,\C^d)$ for all $q \in [1,\infty)$.
            \end{itemize}
            Here $t,t_n$ are independent of $\De$.
    \end{enumerate}
\end{theorem}

Regarding (3), note that while there always exist continuous parameterizations of the roots on the interval $I$, not every 
continuous root parameterization of $P_a$ is the limit of continuous root parameterizations of $P_{a_n}$ as $a_n \to a$.

The proofs of \Cref{t:main1} and \Cref{t:seq} consist of two parts.
First, we prove a general convergence and reparameterization result in the realm of Almgren's $Q$-valued Sobolev functions.
We state it in \Cref{t:Qvt} for Sobolev functions on open bounded intervals with values in $\cA_Q(\R^n)$, 
the space of unordered $Q$-tuples of points in $\R^n$. In \Cref{t:ms}, we even obtain a weaker variant for 
absolutely continuous curves in arbitrary complete metric spaces. 
The natural Lipschitz reparameterization map $t$, which will be defined in this general setting in \Cref{sec:repar}, results from 
a standard construction taken from   
\cite[Section 1.1]{Ambrosio:2008aa}.

The second step in the proof of \Cref{t:main1,t:seq} is essentially a simple combination of \Cref{t:Qvt} with the continuity of the solution map 
provided by \Cref{t:know}; this is carried out in \Cref{sec:proofs}.

\Cref{t:main1,t:seq} close the gap between \Cref{t:know} and the case of \emph{hyperbolic polynomials}, 
where no reparameterization is necessary: 
the polynomial $P_a$ in \eqref{eq:Pa} is called hyperbolic if all roots $\la_j$ are real. 
The coefficient vectors of the hyperbolic polynomials of degree $d$ form a proper 
semialgebraic subset $\Hyp(d) \subseteq \R^d$.
In contrast to the general case, there is canonical 
ordering of the roots by size which induces a locally Lipschitz map $\la_\uparrow : \Hyp(d) \to \R^d$ and the 
superposition operator $\sol_\uparrow : a \mapsto \la_\uparrow \o a$. 
(Thus it is not necessary to work in the space of unordered tuples.)

\begin{theorem}[{\cite{Bronshtein79}, \cite{ParusinskiRainerHyp}, \cite{Parusinski:2024aa}}]\label[t]{t:knowHyp}
    Let $d$ be a positive integer. Then:
    \begin{enumerate}
        \item The map $\sol_\uparrow : C^{d-1,1}(I,\Hyp(d)) \to C^{0,1}(I,\R^d)$ is well-defined and bounded.
        \item The map $\sol_\uparrow  : C^{d}(I,\Hyp(d)) \to W^{1,q}_{\loc}(I,\R^d)$ is continuous for all $q \in [1,\infty)$.
    \end{enumerate}
\end{theorem}

The map $\sol_\uparrow$ is \emph{not} continuous with respect to the $C^{0,1}$ topology on the target space; 
see \cite[Example 1.12]{Parusinski:2024aa}.

For characteristic polynomials of Hermitian and normal matrices, even stronger eigenvalue stability results hold true,
see \cite{Parusinski:2026aa}.

%-----------------------------------------------------------------------------------------------------------------------
\section{Lipschitz reparameterization} \label{sec:repar}
%-----------------------------------------------------------------------------------------------------------------------

This preliminary section is based on \cite[Section 1.1]{Ambrosio:2008aa}.
Throughout this section, $(X,d)$ is a complete metric space. 
Let $(a,b) \subseteq \R$ be a bounded open interval.

%-----------------------------------------------------------------------------------------------------------------------
\subsection{Absolutely continuous curves and their metric slope}
%-----------------------------------------------------------------------------------------------------------------------

A map $f : (a,b) \to X$ is said to belong to $AC^q((a,b),X)$, for $q \in [1,\infty]$, if
there exists $m \in L^q((a,b))$ such that 
\begin{equation} \label{eq:md}
    d(f(x),f(y)) \le \int_x^y m(t)\, dt \quad \text{ whenever } a < x\le y < b.
\end{equation}
Each $f \in AC^q((a,b),X)$ is uniformly continuous and admits a continuous extension to $[a,b]$. 
Furthermore, the limit
\begin{equation}
    |f'|(x) := \lim_{y \to x} \frac{d(f(x),f(y))}{|x-y|}
\end{equation}
exists for $\cL^1$-almost every $x \in (a,b)$ and 
$|f'| \in L^q((a,b))$. (The Lebesgue measure is the only measure we will use; therefore it will not be mentioned explicitly from now on.) 
The function $|f'|$ is called the \emph{metric slope} or \emph{metric derivative} of $f$. 
Note that \eqref{eq:md} holds with $m=|f'|$ and, if $m$ is such that \eqref{eq:md} holds, then 
$|f'|\le m$ almost everywhere in $(a,b)$.

%-----------------------------------------------------------------------------------------------------------------------
\subsection{Lipschitz reparameterization} \label{ssec:Lr}
%-----------------------------------------------------------------------------------------------------------------------

Let $\ep>0$ and $f \in AC^1((a,b),X)$.  
Following \cite[Lemma 1.1.4]{Ambrosio:2008aa}, 
we define reparameterization maps $s = s_{\ep,f}$ and $t=t_{\ep,f}$ as follows.
Setting
\[
    s(x) := \int_a^x (\ep + |f'|(t))\, dt
\]
defines a strictly increasing absolutely continuous function $s  : [a,b] \to [0,L]$, where 
\[
    L = L_{\ep,f} := \ep\, (b-a) + \int_a^b |f'|(t)\, dt.
\]
Let $t  : [0,L] \to [a,b]$ be the inverse of $s$. Then 
$t$ is Lipschitz with $\Lip(t) \le \frac{1}{\ep}$ 
and 
\begin{equation} \label{eq:tos}
    t' \o s = \frac{1}{\ep + |f'|} \quad \text{ almost everywhere in } (a,b). 
\end{equation}
The composite $\hat f := f \o t$ is Lipschitz with $\Lip(\hat f)\le 1$ and 
\begin{equation} \label{eq:nlp}
    |\hat f'|\o s = \frac{|f'|}{\ep+|f'|} \quad \text{ almost everywhere in } (a,b);
\end{equation}
see \cite[Lemma 1.1.4]{Ambrosio:2008aa}.

The right endpoint $L$ of the domain of $t$ depends on $f$. It will be convenient to have the domain of $t$ independent of $f$. 
Thus we extend $t$ to $\R$ by setting $t(y) := a$ if $y < 0$ and $t(y) := b$ if $y > L$.
Then $t  : \R \to [a,b]$ is surjective, increasing, strictly increasing on $(0,L)$, and Lipschitz 
with $\Lip(t) \le \frac{1}{\ep}$.
Clearly, $\hat f = f \o t : \R \to X$ is still Lipschitz with $\Lip(\hat f)\le 1$.

\begin{lemma}  \label[l]{l:rep}
    Let $f,f_n \in AC^1((a,b),X)$, for $n \ge 1$, be such that $|f_n'| \to |f'|$ in $L^1((a,b))$. 
    For fixed $\ep>0$, let $s=s_{\ep,f}$, $s_n=s_{\ep,f_n}$, $t=t_{\ep,f}$, and $t_n=t_{\ep,f_n}$.
    Then:
    \begin{enumerate}
    \item $s_n \to s$ uniformly on $[a,b]$;
    \item $s_n' \to s'$ in $L^1((a,b))$;
    \item $t_n \to t$ uniformly on $\R$.
    \end{enumerate}
\end{lemma}

We will see in \Cref{t:ms} that $t_n' \to t'$ in $L^q_{\loc}(\R)$ for all $q \in [1,\infty)$. 

\begin{proof}
    (1) follows from 
    \begin{align}
        \MoveEqLeft  |s(x)  - s_n(x)| 
        = \Big|\int_a^x |f'|(t) - |f_n'|(t) \, dt \Big| 
        \le \int_a^b \big| |f'|(t) - |f_n'|(t) \big| \, dt.
    \end{align}
    
    (2) is evident, by definition.

    (3) 
    We will only use monotonicity and (1). By (1), $L_n = s_n(b) \to s(b) = L$.
    Define, for  small $\de>0$,  
    \begin{equation}
        g(\de) := \min\Big\{ \min_{x \in [a,b-\de]} (s(x+\de) - s(x)), \min_{x \in [a+\de,b]} (s(x) - s(x-\de))\Big\}.
    \end{equation}
    Then $g(\de)>0$ because $s$ is strictly increasing.
    Fix $\de < \frac{b-a}{2}$.
    By (1), there is $n_0 \ge 1$ such that, for all $n \ge n_0$,
    \begin{equation} \label{eq:C0}
        \|s_{n} - s\|_{C^0([a,b])} < \frac{g(\de)}{2}
    \end{equation}
    and $L_n \in (s(b-\de),L+\de)$.
    
    We first show that, for all $y \in [0,s(b-\de)]$ and $n \ge n_0$, 
    \begin{equation} \label{eq:cl}
       |t(y)-t_n(y)| < \de. 
    \end{equation}
    The remaining $y$ will be considered at the end of the proof.
    
    Let $y \in [0,s(b-\de)]$ and $x := t(y) \in [a,b-\de]$.

    If $x \in [a+\de,b-\de]$, then, by \eqref{eq:C0}, for $n\ge n_0$,
    \begin{align} \label{eq:u1}
        s_{n}(x+\de) &> s(x+\de) -  \tfrac{g(\de)}{2} \ge \frac{s(x+\de)+s(x)}{2} \ge s(x) = y,
        \\ \label{eq:u2}
        s_{n}(x-\de) &< s(x-\de) +  \tfrac{g(\de)}{2} \le \frac{s(x)+s(x-\de)}{2} \le s(x) = y.
    \end{align}
    Therefore, there is a unique $x' \in (x-\de,x+\de)$ such that 
    $s_{n}(x') = y$. As $y < L_n$, 
    we conclude $t_n(y)=x'$ and
    \begin{equation}
        |t(y) - t_{n}(y)|= |x -x'| < \de.
    \end{equation}

    If $x \in [a,a+ \de)$, then \eqref{eq:u1} still holds. Thus $s_n(x+\de)>y$ and $x+\de > t_{n}(y)$.
    Since $x- t_{n}(y) < a+\de - t_{n}(y) \le \de$, 
    we conclude that 
    \begin{equation}
        |t(y) - t_{n}(y)|= |x - t_{n}(y)| < \de. 
    \end{equation}
    Thus \eqref{eq:cl} is proved.

    If $y \in (s(b-\de),L+\de)$, then as $t$ and $t_n$ are $\frac{1}{\ep}$-Lipschitz on $\R$ and $t(L) = b = t_n(L_n)$,
    \begin{align}
        |t(y)-t_n(y)| &\le  |t(y)-t(L)| + |t_n(L_n)-t_n(y)|\le 
        \frac{|y-L| + |L_n - y|}{\ep} 
        \\
                      &\le \frac{2\,|y-L| + |L - L_n|}{\ep} \le \frac{3}{\ep} \max\{\de, s(b)-s(b-\de)\}, 
    \end{align}
    for all $n \ge n_0$.
    Since $s$ is uniformly continuous and 
    $t(y) = t_n(y)$ for all $n \ge n_0$ if $y<0$ or $y\ge L+\de$,  
    we conclude that $t_n \to t$ uniformly on $\R$.
\end{proof}

%-----------------------------------------------------------------------------------------------------------------------
\section{Convergence and reparameterization of \texorpdfstring{$Q$}{Q}-valued Sobolev functions}
%-----------------------------------------------------------------------------------------------------------------------

The main objective of this section is \Cref{t:Qvt}. 

%-----------------------------------------------------------------------------------------------------------------------
\subsection{$Q$-valued Sobolev functions} \label{ssec:Q}
%-----------------------------------------------------------------------------------------------------------------------

We follow \cite{Almgren00} and \cite{De-LellisSpadaro11}.
Let $Q$ be a positive integer.
We consider the space of unordered sets of $Q$ points in $\R^n$ admitting multiplicites, 
\[
    \cA_Q(\R^n) := \Big\{\sum_{i=1}^Q \a{p_i} : p_1,\ldots, p_Q \in \R^n \Big\},
\]
where $\a{p_i}$ denotes the Dirac mass of $p_i \in \R^n$, and endow it with the metric 
\[
    \dd_2\Big(\sum_{i=1}^Q \a{p_i},\sum_{i=1}^Q \a{q_i}\Big) := \min_{\si \in \on{S}_Q} \Big( \sum_{i=1}^Q  \|p_i - q_{\si(i)}\|_2^2\Big)^{1/2}.
\]
The metric space $(\cA_Q(\R^n),\dd_2)$ is complete (see \cite[Remark 0.3]{De-LellisSpadaro11}).

There exists an 
injective Lipschitz map $\De : \cA_Q(\R^n) \to \R^N$ with Lipschitz constant $\Lip(\De)\le 1$, 
where $N=N(n,Q)$.  
The inverse $\De|^{-1}_{\De(\cA_Q(\R^n))}$ is Lipschitz with Lipschitz constant bounded by $C(n,Q)$. 
Moreover, there is a
Lipschitz retraction of $\R^N$ onto $\De(\cA_Q(\R^n))$. See \cite[Theorem 2.1]{De-LellisSpadaro11}.

Almgren used this bi-Lipschitz embedding -- we call $\De$ an \emph{Almgren embedding} -- to define Sobolev spaces of $Q$-valued functions: 
for open $U \subseteq \R^m$ and $q \in [1,\infty)$ we define
\[
    W^{1,q}(U,\cA_Q(\R^n)) := \{f : U \to \cA_Q(\R^n) : \De \o f \in W^{1,q}(U,\R^N)\}.  
\]
For an equivalent intrinsic definition, see \cite[Definition~0.5 and Theorem~2.4]{De-LellisSpadaro11}.
Endowed with the metric
\begin{equation} \label{eq:Almgren}
    (f,g) \mapsto \|\De \o f - \De \o g\|_{W^{1,q}(U,\R^N)},
\end{equation}
$W^{1,q}(U,\cA_Q(\R^n))$ is a complete metric space (where functions that coincide almost everywhere are identified); 
see \cite[Lemma 3.1]{Parusinski:2024ab}.

The topology induced by the metric \eqref{eq:Almgren} is independent of the choice of $\De$. 
This was proved for $\cA_Q(\R^2) \cong \cA_Q(\C)$ in \cite[Theorem 3.11]{Parusinski:2024ab} in the case that $U$ is 
an interval and in \cite[Theorem 8.11]{Parusinski:2026aa} for general $U$. 
This proof generalizes to $\cA_Q(\R^n)$.

Let $I \subseteq \R$ be a bounded open interval and $f \in W^{1,q}(I,\cA_Q(\R^n))$. By \cite[Proposition 1.2]{De-LellisSpadaro11}, there exist $f_1,\ldots, f_Q \in W^{1,q}(I,\R^n)$ 
such that 
\begin{equation} \label{eq:se}
    f = \sum_{i=1}^Q \a{f_i}.
\end{equation}
The following lemma generalizes \cite[Lemma 11.1]{Parusinski:2024ab}; the proof is the same.

\begin{lemma} \label[l]{l:se}
    Let $f \in W^{1,q}(I,\cA_Q(\R^n))$ and $f_1,\ldots, f_Q \in W^{1,q}(I,\R^n)$ such that $f = \sum_{i=1}^Q \a{f_i}$. Then, 
    for almost every $x \in I$,
    \begin{equation}
        |f'|(x) = \|(f_1'(x),\ldots,f_Q'(x))\|_2. 
    \end{equation}
\end{lemma}

In \cite[Theorem 3.11]{Parusinski:2024ab}, it was proved that $f_n \to f$ in $W^{1,q}(I,\cA_Q(\R^2))$ if and only if $f_n \to f$ 
with respect to a semimetric $\dd^{1,q}$ involving the selection \eqref{eq:se}; see \cite[Definition 3.6]{Parusinski:2024ab}. 
As mentioned before, this generalizes to $\cA_Q(\R^n)$.
Together with \Cref{l:se}, it implies the following (cf.\ \cite[Corollary 1.4]{Parusinski:2024ab}):

\begin{lemma} \label[l]{l:ms}
   If $f_n \to f$ in $W^{1,q}(I,\cA_Q(\R^n))$ then $|f_n'| \to |f'|$ in $L^q(I)$.
\end{lemma}

%-----------------------------------------------------------------------------------------------------------------------
\subsection{Strong convergence and reparameterization}
%-----------------------------------------------------------------------------------------------------------------------

Let $(a,b) \subseteq \R$ be a bounded open interval. We will consider functions $f \in W^{1,q}((a,b),\cA_Q(\R^m))$. 
Then $f \in AC^q((a,b),\cA_Q(\R^m))$, by \Cref{l:se} (see also \cite[Remark 1.1.3]{Ambrosio:2008aa} and 
\cite[Proposition 1.2]{De-LellisSpadaro11}). In particular, $f$ is uniformly continuous and extends continuously to $[a,b]$.

\begin{theorem} \label[t]{t:Qvt} 
    Let $f_n \to f$ in $W^{1,q}((a,b),\cA_Q(\R^m))$ for some $q>1$ 
    or, alternatively, $f_n \to f$ in $W^{1,1}((a,b),\cA_Q(\R^m))$ and $f_n(x_0) \to f(x_0)$ for some $x_0 \in [a,b]$. 
    Fix $\ep>0$ and an Almgren embedding $\De : \cA_Q(\R^m) \to \R^N$, and consider $t := t_{\ep,\De \o f}$ and $t_n := t_{\ep,\De \o f_n}$,
    provided by \Cref{ssec:Lr}.
    Then $t_n \to t$ in $W^{1,q}_{\loc}(\R)$ and $f_n \o t_n \to f\o t$ in $W^{1,q}_{\loc}(\R,\cA_Q(\R^m))$, 
    for all $q \in[1,\infty)$.
\end{theorem}

\Cref{t:Qvt} is a consequence of the following proposition, by composition with $\De$.

\begin{proposition} \label[p]{p:Qvt}
    Let $f_n \to f$ in $W^{1,q}((a,b),\R^N)$ for some $q>1$ 
    or, alternatively, $f_n' \to f'$ in $L^1((a,b),\R^N)$ and $f_n(x_0) \to f(x_0)$ for some $x_0 \in [a,b]$. 
    Fix $\ep>0$ and let $t := t_{\ep,f}$ and $t_n := t_{\ep,f_n}$.
    Then $t_n \to t$ in $W^{1,q}_{\loc}(\R)$ and $f_n \o t_n \to f \o t$ in $W^{1,q}_{\loc}(\R,\R^N)$, 
    for all $q \in[1,\infty)$.
\end{proposition}

The rest of this section is dedicated to the proof of \Cref{p:Qvt}.
For simplicity of notation, we write $\hat f := f \o t$ and $\hat f_n := f_n \o t_n$. 
Moreover, $s := s_{\ep,f}$, $s_n := s_{\ep,f_n}$, $L:= s(b)$, and $L_n := s_n(b)$; see \Cref{ssec:Lr}.

We will prove 
\begin{enumerate}
    \item $\hat f_n \to \hat f$ uniformly on $\R$,
    \item $\hat f_n' \to \hat f'$ in $L^q_{\loc}(\R,\R^N)$,
    \item $t_n' \to t'$ in $L^q_{\loc}(\R)$,
\end{enumerate}
for all $q \in [1,\infty)$.
This will imply \Cref{p:Qvt}, because $t_n \to t$ uniformly on $\R$, by \Cref{l:rep}.

Let us check (1). If $f_n \to f$ in $W^{1,q}((a,b),\R^N)$ for some $q>1$, then $f_n \to f$ uniformly on $[a,b]$, 
by Morrey's inequality, and (1) follows 
because $t_n \to t$ uniformly on $\R$ (and $f$ is uniformly continuous on $[a,b]$).
If on the other hand $f_n' \to f'$ in $L^1((a,b),\R^N)$ and $f_n(x_0) \to f(x_0)$ for some $x_0 \in [a,b]$,
then
\begin{align*}
    \| f(x) - f_n(x)\|_2 = \Big\| f(x_0) - f_n(x_0) + \int_{x_0}^x f'(u) - f_n'(u)\, du\Big\|_2 \to 0
\end{align*}
uniformly for $x \in [a,b]$ and (1) follows as before.

Next we prove (2).
Since $t$ is increasing, $\hat f = f \o t$ is differentiable 
almost everywhere and the chain rule holds (see e.g. \cite[Corollary 3.50]{Leoni09}),
\begin{equation} \label{eq:hld1} 
    \hat f' = (f' \o t)\cdot t' = (f' \o t) \cdot \frac{1}{s' \o t}
    = \frac{f'}{\ep + \|f'\|_2} \o t
\end{equation}
almost everywhere in $(0,L)$ and $\hat f' = 0$ in $\R \setminus [0,L]$.
Similarly, 
\begin{equation} \label{eq:hld2} 
    \hat f_n' =  \frac{f_n'}{\ep + \|f_n'\|_2} \o t_n
\end{equation}
almost everywhere in $(0,L_n)$ and $\hat f' = 0$ in $\R \setminus [0,L_n]$.

In \Cref{l:cim}, we will prove that $\hat f_n' \to \hat f'$ in measure on $\R$.
Since $\|\hat f_n'\|_2 \le 1$ almost everywhere in $\R$ and hence the family $\{\hat f_n'\}$ is 
uniformly integrable on each bounded interval, this implies that $\hat f_n' \to \hat f'$ in $L^q_{\loc}(\R,\R^N)$ for all $q \in [1,\infty)$, 
by Vitali's convergence theorem. Thus (2) is proved.

\begin{lemma} \label[l]{l:cim}
    $\hat f_n' \to \hat f'$ in measure on $\R$.
\end{lemma}

\begin{proof}

Since $f_n \to f$ in $W^{1,1}((a,b),\R^N)$, after passing to a subsequence, 
$f_n' \to f'$ almost everywhere in $(a,b)$. It follows that 
\begin{equation}
    \frac{f_n'}{\ep + \|f_n'\|_2} \to \frac{f'}{\ep + \|f'\|_2} \quad \text{ almost everywhere in } (a,b).
\end{equation}

By the generalized dominated convergence theorem (see e.g.\ \cite[p.~59]{Folland99}), we conclude that 
\begin{equation} \label{eq:l1c}
    \frac{f_n'}{\ep + \|f_n'\|_2 } \to \frac{f'}{\ep + \|f'\|_2 }\quad \text{ in } L^1((a,b),\R^N). 
\end{equation}
We proved \eqref{eq:l1c} for a subsequence, but this implies its validity for the original sequence.

    Let us put 
    \[
        \al_n := \frac{f_n'}{\ep + \|f_n'\|_2 }, \quad \al := \frac{f'}{\ep + \|f'\|_2 }.
    \]
    Then $\al_n \to \al$ in $L^1((a,b),\R^N)$, by \eqref{eq:l1c}. 
    We have $\hat f' = \al \o t$ almost everywhere in $(0,L)$, by \eqref{eq:hld1},
    and $\hat f'(y) = 0$ if $y<0$ or $y>L$.
    By \eqref{eq:hld2}, $\hat f_n' = \al_n \o t_n$ almost everywhere in $(0,L_n)$
    and $\hat f_n'(y) = 0$ if $y<0$ or $y>L_n$.

    Let $\rh,\si>0$ be fixed.
    We have to show that there exists $n_0 \ge 1$ such that 
    \begin{equation} \label{eq:ats}
        |\{y \in \R : \|\hat f_n'(y) - \hat f'(y)\|_2 \ge \rh\}| \le \si, \quad n \ge n_0.
    \end{equation}
    
    We claim that it suffices to show \eqref{eq:ats} for $y \in (0,L-\frac{\si}{4})$ (assuming of course that $\si$ is small enough).
    Indeed,
    if $y < 0$ then $\hat f_n'(y) = \hat f'(y) = 0$ for all $n$. 
    On the other hand, since $L_{n} \to L$, there is $n_1 \ge 1$ such that for all $n \ge n_1$ 
    we have $L_{n} \in (L - \frac{\si}{4},L + \frac{\si}{4})$ and hence
    $\hat f_n'(y) = \hat f'(y) = 0$ for all $y> L +\tfrac{\si}{4}$. 
    Thus, for $n \ge n_1$, the left-hand side of \eqref{eq:ats} equals
    \[
        |\{y \in (0,L + \tfrac{\si}{4}) : \|\hat f_n'(y) - \hat f'(y)\|_2 \ge \rh\}|
    \]
    so that, in order to establish \eqref{eq:ats}, it suffices to show
    \begin{equation} \label{eq:ats2}
        |\{y \in (0,L - \tfrac{\si}{4}) : \|\al_n(t_{n}(y)) - \al(t(y))\|_2 \ge \rh\}| \le \frac{\si}{2},
    \end{equation}
    for all sufficiently large $n$.
    Note that, for all $n \ge n_1$, the restriction of $t_{n}$ to $[0,L-\tfrac{\si}{4}]$ is invertible 
    with inverse $s_{n}|_{[a,t_{n}(L - \tfrac{\si}{4})]}$.

    Since $s_{n}' \to s'$ in $L^1((a,b))$ (see \Cref{l:rep}), the set $\{s_{n}' : n\ge 1\}\cup \{s'\}$ is uniformly integrable, i.e., 
    there exists $\ta>0$ such that, for all $n \ge 1$, 
    \begin{equation} \label{eq:aui}
        |s_{n}(E)| =  \int_E s_{n}'(x)\, dx \le \frac{\si}{6} \quad \text{ and } \quad |s(E)|\le \frac{\si}{6}  \quad \text{ if } |E|\le \ta.
    \end{equation}
    
    By Lusin's theorem, there exists a compact $K \subseteq [a,b]$ with $|[a,b] \setminus K| \le \ta$  
    such that $\al|_K$ is continuous. By the Tietze--Dugundji theorem, there exists a continuous function $g : [a,b]\to \C^d$ 
    such that $g|_K = \al|_K$.

    For $y \in (0,L- \tfrac{\si}{4})$, consider
    \begin{align*}
        \al_n(t_{n}(y)) - \al(t(y)) &= A_n(y) + B_n(y) + C_n(y) + D(y),   
    \end{align*}
    where 
    \begin{align*}
        A_n(y) &:= \al_n(t_{n}(y)) - \al(t_{n}(y)),
        \\
        B_n(y) &:= \al(t_{n}(y)) - g(t_{n}(y)),
        \\
        C_n(y) &:= g(t_{n}(y)) - g(t(y)),
        \\
        D(y) &:= g(t(y)) - \al(t(y)).
    \end{align*}
    It follows that, if 
\begin{equation}
    \|A_n(y)\|_2 < \tfrac{\rh}{2},\quad  B_n(y) = D(y)=0,\quad \|C_n(y)\|_2 < \tfrac{\rh}{2},
\end{equation}
then 
    \begin{equation}
         \|\al_n(t_{n}(y)) - \al(t(y)) \|_2 < \rh.
    \end{equation}
    
    We have $\{y \in (0,L-\tfrac{\si}{4}) : \|A_n(y)\|_2 \ge \tfrac{\rh}{2}\} = s_{n}(F_n)$, where 
    \[
        F_n := \{x \in (a,t_{n}(L-\tfrac{\si}{4})) : \|\al_n(x) - \al(x)\|_2 \ge \tfrac{\rh}{2}\}.
    \]
    Since $\al_n \to \al$ in $L^1((a,b),\R^N)$ and
    \begin{equation}
        \int_a^b \|\al_n(x) - \al(x)\|_2 \, dx \ge \int_{F_n} \|\al_n(x) - \al(x)\|_2 \, dx \ge \tfrac{\rh}{2} \, |F_n|,
    \end{equation}
    we have $|F_n| \le \ta$ for all sufficiently large $n$, so that by \eqref{eq:aui}, 
    \begin{equation}
        |\{y \in (0,L-\tfrac{\si}{4}) : \|A_n(y)\|_2 \ge \tfrac{\rh}{2}\}| = |s_{n}(F_n)| 
        \le \frac{\si}{6}.
    \end{equation}

    Next we observe that $B_n(y) = 0$ if $y \in s_{n}(K)$ and that $|s_{n}([a,b]\setminus K)| \le \tfrac{\si}{6}$, 
    by \eqref{eq:aui}.
    Similarly, $D(y) = 0$ if $y \in s(K)$ and $|s([a,b]\setminus K)| \le \tfrac{\si}{6}$.

    Finally, $\|C_n(y)\|_2 < \tfrac{\rh}{2}$ uniformly for all $y$, if $n$ is sufficiently large, in view of \Cref{l:rep}.

    We conclude that, there is $n_2\ge n_1$ such that for all $n \ge n_2$,
    \begin{align*}
        |\{y \in (0,L- \tfrac{\si}{4}) : \|\al_n(t_{n}(y)) - \al(t(y)) \|_2 \ge \rh \}| 
        \le \frac{\si}{6} + \frac{\si}{6} +\frac{\si}{6}=\frac{\si}{2}
    \end{align*}
    which proves \eqref{eq:ats2}.
\end{proof}

The following lemma follows by the same arguments.

\begin{lemma} \label[l]{l:tim}
   $t_n' \to t'$ in measure on $\R$. 
\end{lemma}

\begin{proof}
    We have  
    \[
        t' = \frac{1}{s' \o t}
    \]
    almost everywhere in $(0,L)$ and $t' = 0$ in $\R \setminus [0,L]$.
    Similarly,
    \[
        t_n' = \frac{1}{s_n' \o t_n}
    \]
    almost everywhere in $(0,L_n)$ and $t' = 0$ in $\R \setminus [0,L_n]$.

    Let $\rh,\si>0$ be fixed.
    We have to show that there exists $n_0 \ge 1$ such that 
    \begin{equation} \label{eq:ts}
        |\{y \in \R : |t_{n}'(y) - t'(y)| \ge \rh\}| \le \si, \quad n \ge n_0.
    \end{equation}
    As in the proof of \Cref{l:cim}, we see that it suffices to show \eqref{eq:ts} for $y \in (0,L-\frac{\si}{4})$.
    
    Since $s',s_{n}'\ge \ep$ almost everywhere in $(a,b)$,  
    \begin{equation} \label{eq:li}
        |t_{n}' - t'| = \Big| \frac{1}{s_{n}'\o t_n} - \frac{1}{s' \o t} \Big| 
        \le \ep^{-2} | s_n'\o t_n - s'\o t |
    \end{equation}
    almost everywhere in $(0,L-\frac{\si}{4})$, for sufficiently large $n$.

    In view of \eqref{eq:li},
    setting $\al_n := s_n'$ and $\al := s'$, the assertion of the lemma follows from \eqref{eq:ats2}.
\end{proof}

\Cref{l:tim} implies $t_n' \to t'$ in $L^q_{\loc}(\R)$ for all $q \in [1,\infty)$, because $|t_n'| \le \frac{1}{\ep}$ almost everywhere in $\R$.
Thus (3) is proved.

This ends the proof of \Cref{p:Qvt}.

%-----------------------------------------------------------------------------------------------------------------------
\subsection{Metric slope convergence and reparameterization}
%-----------------------------------------------------------------------------------------------------------------------

The proof of \Cref{t:Qvt} yields the following variant for absolutely continuous curves in an 
arbitrary complete metric space $(X,d)$.

\begin{theorem} \label[t]{t:ms}
    Let $f,f_n \in AC^1((a,b),X)$ and $|f_n'| \to |f'|$ in $L^1((a,b))$.
    Fix $\ep>0$ and
    let $\hat f := f \o t$ and $\hat f_n := f_n \o t_n$, where $t := t_{\ep,f}$ and $t_n := t_{\ep,f_n}$.
    Then $t_n \to t$ in $W^{1,q}_{\loc}(\R)$ and $|\hat f_n'| \to |\hat f'|$ in $L^q_{\loc}(\R)$, 
    for all $q \in[1,\infty)$.
\end{theorem}

\begin{proof} 
    We use the notation of the previous section. 
    \Cref{l:rep,l:tim} imply $t_n \to t$ in $W^{1,q}_{\loc}(\R)$, for all $q \in [1,\infty)$, as above. 
    By \eqref{eq:nlp}, 
    \[
        |\hat f'| \o s = \frac{|f'|}{\ep + |f'|} \quad \text{ almost everywhere in } (a,b).
    \]
    Consequently (note that $s$ and $t$ have the Lusin $N$-property),  
    \[
        |\hat f'|  = \frac{|f'|}{\ep + |f'|} \o t \quad \text{ almost everywhere in } (0,L).
    \]
    Since $\hat f$ is locally constant on $\R \setminus [0,L]$, we have $|\hat f'|(y)=0$ for 
    $y \in \R \setminus [0,L]$. 
    Analogously,
    \[
        |\hat f_n'|  = \frac{|f_n'|}{\ep + |f_n'|} \o t \quad \text{ almost everywhere in } (0,L_n),
    \]
    and
    $|\hat f_n'|(y)=0$ for $y \in \R \setminus [0,L_n]$.

    Similarly as \eqref{eq:l1c}, we find that
    \begin{equation}
        \frac{|f_n'|}{\ep + |f_n'|} \to \frac{|f'|}{\ep + |f'|} \quad \text{ in } L^1((a,b)).
    \end{equation}
    The proof of \Cref{l:cim} (with $\al_n=\frac{|f_n'|}{\ep + |f_n'|}$ and $\al =\frac{|f'|}{\ep + |f'|}$) 
    gives that $|\hat f_n'| \to |\hat f'|$ in measure on $\R$. 
    Consequently, $|\hat f_n'| \to |\hat f'|$ in $L^q_{\loc}(\R)$, 
    for all $q \in[1,\infty)$, because $|\hat f_n'| \le 1$ almost everywhere in $\R$.
\end{proof}

%-----------------------------------------------------------------------------------------------------------------------
\section{Proof of \texorpdfstring{\Cref{t:main1} and \Cref{t:seq}}{Theorem 1.2 and Theorem 1.3}} \label{sec:proofs}
%-----------------------------------------------------------------------------------------------------------------------

Let $d \in \N_{\ge 1}$, $\ep>0$, and $I \subseteq \R$ a bounded open interval.
For $a \in C^{d-1,1}(\ol I,\C^d)$ consider $\sol(a) = \La_a$ we belongs to $W^{1,q}(I,\cA_d(\C))$, for all $q \in [1,\frac{d}{d-1})$; see \Cref{t:know}. 
Fix an Almgren embedding $\De : \cA_d(\C) \to \R^N$.
We define (cf.\ \Cref{ssec:Lr}) 
\[
    t_a := t_{\ep,\De \o \La_a} : \R \to \ol I
\]
and thus obtain the reparameterization map $t : C^{d-1,1}(\ol I,\C^d) \to \Lip(\R,\ol I)$.
As seen in  \Cref{ssec:Lr}, $t$ has the properties (1)--(3) listed in \Cref{t:main1}.
The continuity of the map \eqref{eq:rs}, for all $q \in [1,\infty)$, is a consequence of \Cref{t:seq}.

Let us prove \Cref{t:seq}.
Let $a_n \to a$ in $C^d(\ol I,\C^d)$. Then $\La_{a_n} \to \La_a$ uniformly on $\ol I$, by \cite[Corollary 6.5]{Parusinski:2024ab}, 
and in $W^{1,q}(I,\cA_d(\C))$ for all $q \in [1,\frac{d}{d-1})$, by \Cref{t:know}. 
By \Cref{l:rep}, $t_{a_n} \to t_a$ and hence also $\La_{a_n} \o t_{a_n} \to \La_a \o t_a$ uniformly on $\R$.
\Cref{t:Qvt} implies that $t_{a_n} \to t_a$ in $W^{1,q}_{\loc}(\R)$ and $\La_{a_n} \o t_{a_n} \to \La_{a} \o t_{a}$ in $W^{1,q}_{\loc}(\R,\cA_d(\C))$, 
for all $q \in [1,\infty)$.

To prove (3) of \Cref{t:seq} note that, in the respective setup, 
$\la_n \to \la$ in $W^{1,q}(I,\C^d)$ for all $q \in [1,\frac{d}{d-1})$, by \cite[Theorem 1.6]{Parusinski:2024ab}.
Consequently, the statement follows from \Cref{p:Qvt}. Here $t_n := t_{\ep,\la_n}$ and $t := t_{\ep,\la}$.

The proofs of \Cref{t:main1} and \Cref{t:seq} are complete.

%---------------------------------------------------------------------------------------------
\subsection*{Acknowledgement}
%---------------------------------------------------------------------------------------------

This research was funded in whole or in part by the Austrian Science Fund (FWF) DOI 10.55776/PAT1381823.
For open access purposes, the authors have applied a CC BY public copyright license to any author-accepted 
manuscript version arising from this submission.

The authors acknowledge the use of Claude Sonnet 5 for suggesting an argument that led to the proof 
of a precursor of \Cref{l:cim}.
The authors verified this argument and adapted it to the present setting, and take 
full responsibility for the content of the paper.

%\bibliography{../../references/biblio}
%\bibliographystyle{amsalpha}

\def\cprime{$'$}
\providecommand{\bysame}{\leavevmode\hbox to3em{\hrulefill}\thinspace}
\providecommand{\MR}{\relax\ifhmode\unskip\space\fi MR }
\providecommand{\MRhref}[2]{%
  \href{http://www.ams.org/mathscinet-getitem?mr=#1}{#2}
}
\providecommand{\href}[2]{#2}

\end{document}